\documentclass{imanum}
\usepackage{graphicx}

\usepackage{lipsum}
\usepackage{algorithmic}
\usepackage{epstopdf}
\usepackage{geometry}
\usepackage{epsfig}
\usepackage{amsfonts}
\usepackage{algorithm}
\usepackage[figuresright]{rotating}
\usepackage{amscd}
\usepackage{mathrsfs}
\usepackage{booktabs,longtable}
\usepackage{indentfirst}
\usepackage{subfigure}
\usepackage{amstext}
\usepackage{amssymb}
\usepackage{fancyhdr}
\usepackage{amsmath}
\usepackage{cleveref}

\begin{document}

\title{Splitting-based randomized iterative methods for solving indefinite least squares problem}
\shorttitle{Splitting-based randomized iterative methods}

\author{%
{\sc
Yanjun Zhang 
and Hanyu Li\thanks{Corresponding author. Email: lihy.hy@gmail.com or hyli@cqu.edu.cn}
 } \\[2pt]
College of Mathematics and Statistics, Chongqing University, Chongqing 401331, P.R. China 
}
\shortauthorlist{Y. J. Zhang \emph{et al.}}

\maketitle

\begin{abstract}
{The indefinite least squares (ILS) problem is a generalization of the famous linear least squares problem. It minimizes an indefinite quadratic form with respect to a signature matrix. For this problem, we first propose an impressively simple and effective splitting (SP) method according to its own structure 
and prove that it converges `unconditionally' for any initial value. Further, to avoid  implementing some matrix multiplications and calculating the inverse of large matrix and considering the acceleration and efficiency of the randomized strategy, we develop two randomized iterative methods on the basis of the SP method as well as the randomized Kaczmarz, Gauss-Seidel and coordinate descent methods, and describe their convergence properties. 
Numerical results 
show that our 
three methods all have quite decent performance in both computing time and iteration numbers compared with the latest 
iterative method of the ILS problem, and also demonstrate that the two randomized methods indeed yield significant acceleration in term of computing time.}
{indefinite least squares problem; splitting method; randomized method; Kaczmarz; Gauss-Seidel; coordinate descent.}
\end{abstract}

\section{Introduction}
\label{sec;introduction}
 The indefinite least squares (ILS) problem was first proposed in 
\citet{chandrasekaran1998stable}, whose specific form is as follows:
\begin{align}
\text{ILS:} \quad \min \limits_{x\in\mathbb{R}^n}(b-Ax)^TJ(b-Ax),  \label{Sec11}
\end{align}
where $A\in\mathbb{R}^{m\times n}$ with $m\geq n$, $b\in\mathbb{R}^m$, and $J$ is the signature matrix defined as
\begin{align}
J=\left[\begin{array}{cc}
I_p& 0 \\
0 & -I_q
\end{array}\right], \notag  \quad p+q=m.
\end{align}
Here and in the sequel, $G^T$ denotes the transpose of $G$ and $I_t$ is the identity matrix of dimension $t$.
Obviously, the ILS problem will reduce to the standard linear least squares problem when $q=0$. However, for $pq>0$, the problem (\ref{Sec11}) is to minimize an indefinite quadratic form  with respect to the signature matrix $J$ and its 
normal equation is:
\begin{align}
  A^TJAx=A^TJb.  \label{Sec12}
\end{align}
Considering that the Hessian matrix of the problem (\ref{Sec11}) is $2A^TJA$, the ILS problem has a unique solution if and only if
 \begin{align}
 A^TJA  \quad \text{is symmetric and positive definite (SPD).} \label{Sec13}
\end{align}
Throughout this paper, we assume that the above condition is always true.

The ILS problem has found many applications in some fields, such as the total least squares problems \citep[see, e.g.,][]{golub1980analysis, van1991total} and $H^{\infty}$ smoothing \citep[][]{hassibi1993recursive}. Extensive works on computations, perturbation analysis, and applications of this problem have been published \citep[see, e.g.,][]{chandrasekaran1998stable,bojanczyk2003solving,xu2004backward,liu2011preconditioned, liu2013incomplete,liu2014block,li2014mixed,li2018partial,diao2019backward,song2020ussor,bojanczyk2021algorithms}. In this paper, we mainly focus on its numerical methods. For the small and dense ILS problem (\ref{Sec11}), 
\citet{chandrasekaran1998stable} designed a stable direct method called the QR-Cholesky method, which 
first performs the QR factorization of $A$, i.e., $A=QR$ with $Q^TQ=I$ and $R$ being a upper triangular matrix, and then solves $ (Q^TJQ)y=Q^TJb$ by using the Cholesky factorization. Finally, the solution is returned by $x=R^{-1}y$. 
Later, 
\citet{bojanczyk2003solving} devised a method with a lower operation count than the QR-Cholesky method 
by using the hyperbolic QR factorization. After that, \citet{xu2004backward} proposed to apply the hyperbolic QR factorization to the normalized matrix $A$ to make the algorithm be backward stable.
More recently, a unified analysis of the above three methods was made in \citet{bojanczyk2021algorithms}. 
For the large and sparse ILS problem, the direct methods are no longer feasible and hence it is necessary to introduce the iterative methods. Specifically, the preconditioned conjugate gradient methods were first considered in \citet{liu2011preconditioned} and \citet{liu2013incomplete}. Then, 
\citet{liu2014block} investigated the block SOR method with a relaxation parameter, which was further improved recently by
\citet{song2020ussor} who presented the USSOR method including two parameters. 

\subsection{Motivation and contributions}
In \citet{song2020ussor}, the author first transformed the normal equation (\ref{Sec12}) into a larger linear system and then introduced the USSOR  method with two parameters into the new system for solving the problem (\ref{Sec11}). The method is convergent 
only under certain conditions, 
and 
their numerical results show that different parameters will lead to different results and it is 
difficult to determine the 
optimal parameters 
for large-scale problems. Instead, 
we propose a splitting (SP) method without parameter for solving the ILS problem (\ref{Sec11}) 
by fully exploiting the structure of the problem itself and show that the new method 
converges `unconditionally' in theory. For numerical results, it is also uniformly superior to the USSOR method in \citet{song2020ussor}.

Considering that 
the SP method 
needs to compute matrix products and inverse and the cost is prohibitive for large-scale matrices, 
we transform our splitting iterative scheme
into two individual consistent linear subsystems. Then, 
the randomized Kaczmarz (RK) method \citep[][]{strohmer2009randomized} and the randomized Gauss-Seidel (RGS) method \citep[][]{leventhal2010randomized} are applied to solve each subsystem and hence we propose the splitting-based RK-RGS (SP-RK-RGS) method for solving the ILS problem (\ref{Sec11}). 
It is interesting that the proposed joint randomized iterative method only accesses two columns of matrix in each iteration. 
Furthermore, the 
method also allows for opportunities to execute in parallel. 

Another interesting finding is that when the indices of the two random columns in each iteration of the SP-RK-RGS method are the same, the joint randomized iterative update will reduce to the randomized coordinate descent (RCD) update \citep[][]{leventhal2010randomized}. Inspired by this result, we 
design a splitting-based sampling coordinate descent (SP-SCD) method for the ILS problem (\ref{Sec11}), which can accelerate the SP-RK-RGS method.

\subsection{Outline}
The paper is organized as follows. We propose the SP method and present its convergence analysis in \Cref{sec:Splitting method}. The SP-RK-RGS and SP-SCD methods and their convergence analysis are provided in \Cref{sec:randomized method,sec:randomized sampling method}
, respectively. We report extensive numerical results in \Cref{sec:experiments}. Finally, the concluding remarks of the whole paper are given in \Cref{sec:conclusions}.
\subsection{Notation}
 For a matrix $G=(G_{(i, j)})\in \mathbb{R}^{m\times n}$, $G^{(i)}$, $G_{(j)}$, $\text{rank}(G)$, $\sigma_{\max}(G)$, $\sigma_{\min}(G)$, $\|G\|_2$, $\|G\|_F$, and $G^{\upsilon}$ denote its $i$th row, $j$th column, rank, largest singular value, smallest nonzero singular value, spectral norm, Frobenius norm, and the restriction onto the row indices in the set $\upsilon$, respectively. If $G$ is a square matrix, i.e., $m=n$,  $\lambda(G)$ stands for an eigenvalue of $G$, and $\rho(G)=\max\limits_{1\leq i\leq n}|\lambda_{i}(G)|$ represents its spectral radius; if $G\in \mathbb{R}^{n\times n}$ is SPD, we define the energy norm of any vector $x\in \mathbb{R}^{n}$ as $\| x\|_G:=\sqrt{x^TGx}$.
 For a vector $z\in \mathbb{R}^{n}$, $z^{(j)}$ represents its $j$th entry. In addition, we use $e_{(j)}$,  $ \mathbb{E}^{k-1}$, and $ \mathbb{E}$ to denote the $j$th column of the identity matrix $I$, the conditional expectation conditioned on the first $k-1$ iterations, and the full expected value, respectively, and let $[m]:=\{1, 2, 3, \ldots, m\}$ for an integer $m\geq 1$. Finally, we partition $A$ and $b$ in the ILS problem (\ref{Sec11}) as
 \begin{align}
A=\left[\begin{array}{cc}
A_1\\
A_2
\end{array}\right], \quad
b=\left[\begin{array}{cc}
b_1\\
b_2
\end{array}\right],\label{Sec121}
\end{align}
where $A_1\in\mathbb{R}^{p\times n}$, $A_2\in\mathbb{R}^{q\times n}$, $b_1\in\mathbb{R}^{p}$ and $b_2\in\mathbb{R}^{q}$.
\section{SP method for the ILS problem}
\label{sec:Splitting method}

From the partition form of $A$ defined in (\ref{Sec121}), we have
 \begin{align}
A^TJA=A_1^TA_1-A_2^TA_2,\notag 
\end{align}
and hence (\ref{Sec12}) can be equivalently rewritten as the following linear system
 \begin{align}
(A_1^TA_1-A_2^TA_2)x=A^TJb,  \notag
\end{align}
which can be
rewritten further as
 \begin{align}
 A_1^TA_1x=A_2^TA_2 x+A^TJb. \notag
\end{align}
Note that $ A^TJA $ is SPD as described in (\ref{Sec13}), 
so is $A_1^TA_1$. Thus, we can devise an update formula 
as follows:
\begin{align}
x_{k+1}= ( A_1^TA_1)^{-1}A_2^TA_2 x_{k}+( A_1^TA_1)^{-1}A^TJb, \quad k=0, 1, \ldots. \label{Sec23}
\end{align}
Therefore, we construct the SP method, i.e., \Cref{alg:splitting}.
\begin{algorithm}
\caption{SP method for the ILS problem (\ref{Sec11}).}
\label{alg:splitting}
\begin{algorithmic}[1]
\STATE{Input: $A$, $J$, $b$, and the initial estimate $x_0$. }
\STATE{Set $P=( A_1^TA_1)^{-1}.$}
\STATE{Set $B=PA_2^TA_2. $}
\STATE{Set $c=PA^TJb.$ }
\FOR{$k=0, 1, 2, \ldots $ until convergence,}
\STATE{Update $x_{k+1}=Bx_{k}+c.$ }
\ENDFOR
\end{algorithmic}
\end{algorithm}

\begin{remark}\label{remark-eff: SP}
The derivation and iterative scheme of the SP method are very simple and concise. Ordinarily, it should be discovered earlier or should have bad performance. However, we didn't find it in any literature on ILS problem and indeed find that its   performance 
in solving the problem (\ref{Sec11}) is quite encouraging; see the numerical results in 
\Cref{sec:experiments} for details. 
\end{remark}
\begin{remark}\label{remark-cost: SP}
In \Cref{alg:splitting}, computing step 2 to step 4 needs operation counts of about $pn^2+2n^3$, $qn^2+2n^3$, and $2mn+2n^2-2n$, respectively, and hence gives the total counts of about $mn^2+4n^3+2mn+2n^2-2n$. Updating $x_{k+1}$ in step 6 requires about $2n^2$ operation counts in each iteration and hence the total operation counts of the SP method are about
$$mn^2+4n^3+2mn+2n^2-2n+2n^2 \cdot T_{\text{SP}},$$
where $T_{\text{SP}}$ is the number of iterations. This cost is almost the same as the one of the USSOR method in \citet{song2020ussor}; see \Cref{subsec:Computational complexities} for the specific cost of the USSOR method. However, the SP method 
performs better in computing time and iteration numbers, which is confirmed by the numerical experiments in
\Cref{sec:experiments}. The phenomenon also appears in the SP-SCD method. Some possible reasons are given in \Cref{subsec:Computational complexities}.
\end{remark}

Now, we present the convergence analysis of the SP method.
 \begin{theorem} \label{thm:split}
For the ILS problem (\ref{Sec11}), 
the SP method, i.e., \Cref{alg:splitting}, converges for any initial vector $x_{0}$.
\end{theorem}

\begin{proof}
Since $A^TJA= A_1^TA_1-A_2^TA_2$ and $A_1^TA_1$ are SPD, it is easy to see that
$$(A_1^TA_1)^{-\frac{1}{2}}(A_1^TA_1-A_2^TA_2)(A_1^TA_1)^{-\frac{1}{2}}=I-(A_1^TA_1)^{-\frac{1}{2}} A_2^TA_2 (A_1^TA_1)^{-\frac{1}{2}}$$
is also SPD, which implies that the eigenvalues of $(A_1^TA_1)^{-\frac{1}{2}} A_2^TA_2 (A_1^TA_1)^{-\frac{1}{2}}$ satisfy
$$0\leq\lambda\left((A_1^TA_1)^{-\frac{1}{2}} A_2^TA_2 (A_1^TA_1)^{-\frac{1}{2}}\right)<1.$$
Thus, we have that
the spectral radius of the iteration matrix of the SP method is less than 1, i.e.,
  \begin{align}
\rho\left(\left(A_1^TA_1\right) ^{-1}A_2^TA_2\right)<1,\notag
\end{align}
which concludes the convergence of the SP method for any initial vector $x_{0}$.
\end{proof}

\begin{remark}\label{re-sp}
\Cref{thm:split} indicates that the SP method can be seen as an `unconditionally' convergent iterative method. Of course, the initial acknowledged condition (\ref{Sec13}) needs to be satisfied. 
\end{remark}

\section{Splitting-based RK-RGS method for the ILS problem}
\label{sec:randomized method}
In 
the SP method, 
we need to implement the matrix multiplication $ A_1^TA_1$ and compute its inverse $( A_1^TA_1)^{-1}$ and $( A_1^TA_1)^{-1}A_2^TA_2$. For the large-scale ILS problem (\ref{Sec11}), and especially when $p \gg q$, the cost is prohibitive. The extreme case on $p,q$ appears in Minkowski spaces \citep[][]{vsego2009two}, where $p=m-1$ and $q=1$. 
To reduce the cost, we will transform (\ref{Sec23}) into two subsystems and then adopt the RK and RGS methods to solve them.

We begin by briefly reviewing the RK and RGS methods, which play a foundational role in our proposed method. 

\subsection{RK method for linear problem}
\label{subsec:RK method}
Consider the consistent linear system
 \begin{align}
 X\beta=y, \label{Sec3.11}
\end{align}
where $X \in\mathbb{R}^{t\times l}$ is a full row rank matrix, 
$y\in\mathbb{R}^{t}$, and $\beta$ is the $l$-dimentional unknown vector. Starting from a vector $\beta_{0}$, the RK method 
repeats the following two steps in each iteration. First, it chooses a row $i_{k}$ of $X$ with probability proportional to the square of its Euclidean norm, i.e.,
 \begin{align}
\operatorname{Pr}\left(\mathrm{row}=i_{k}\right)=\frac{\| X^{(i_{k})}\|_{2}^{2}}{\| X\|_{F}^{2}}.\notag
\end{align}
Then, it projects the current iteration orthogonally onto the solution hyperplane of that row, i.e.,
 \begin{align}
\beta_{k+1}=\beta_{k}+\frac{ y^{(i_{k})}-X^{(i_{k})}\beta_k}{ \| X^{\left(i_{k}\right)} \|_{2}^{2}}( X^{(i_{k})})^{T}.  \notag  
\end{align}
This randomized method was first investigated in \cite{strohmer2009randomized}. Then, 
\cite{ma2015convergence} showed that it converges linearly to the least Euclidean norm solution
 \begin{align}
\beta_{LN}=X^T(X X^T)^{-1}y  \label{Sec3.13}
\end{align}
of (\ref{Sec3.11}). Specifically, the iteration $\beta_k$ satisfies the following expected linear rate:
 \begin{align}
\mathbb{E}\left[\left\|\beta_{k}-\beta_{LN}\right\|^{2}_2\right]\leq \left(1-\frac{\sigma_{\min}^{2}(X)}{\|X\|_{F}^{2}}\right)^{k}  \left\|\beta_{0}-\beta_{LN}\right\|_{2}^{2}. \label{Sec3.14}
\end{align}
Later, this convergence rate was further accelerated by using various strategies including block strategies \citep[see, e.g.,][]{needell2014paved,necoara2019faster, du2020randomized,zhang2021block}, greedy strategies \citep[see, e.g.,][]{nutini2016convergence,bai2018greedy, niu2020greedy, gower2021adaptive,Zhang2022MK}, and others 
\citep[see, e.g.,][]{lin2015extended, liu2016accelerated,jiao2017preasymptotic}. In addition, the RK method was also extended to many other problems such as the inconsistent problems \citep[see, e.g.,][]{zouzias2013randomized, wang2015randomized}, the ridge regression problems \citep[see, e.g.,][]{hefny2017rows,liu2019variant}, the feasibility problems \citep[see, e.g.,][]{de2017sampling, morshed2020accelerated,morshed2021sampling}, etc.

\subsection{RGS method for linear problem}
\label{subsec:RGS method}
Consider the consistent linear system
 \begin{align}
 X\beta=y, \label{Sec3.21}
\end{align}
where $X \in\mathbb{R}^{t\times l}$ is a full column rank matrix, 
$y\in\mathbb{R}^{t}$, and $\beta$ is the $l$-dimentional unknown vector. From an initial vector $\beta_{0}$, the RGS method relies on columns rather than rows in each iteration. 
Specifically, it first chooses a column $j_{k}$ of $X$ with probability proportional to the square of its Euclidean norm, i.e.,
 \begin{align}
\operatorname{Pr}\left(\mathrm{column}=j_{k}\right)=\frac{\| X_{(j_{k})}\|_{2}^{2}}{\| X\|_{F}^{2}}. \label{Sec3.21.1}
\end{align}
Then, it updates the iteration
 \begin{align}
\beta_{k+1}=\beta_{k}+\frac{X_{(j_{k})}^T \left(y-X\beta_k\right) }{ \| X_{\left(j_{k}\right)} \|_{2}^{2}}e_{(j_k)}.  \notag 
\end{align}
This method was proposed by Leventhal and Lewis
\citep[][]{leventhal2010randomized}. They
also proved that the RGS method converges to the unique solution
 \begin{align}
\beta^{\star}=(X^TX)^{-1}X^Ty  \label{Sec3.23}
\end{align}
of (\ref{Sec3.21}) with the following linear rate:
 \begin{align}
\mathbb{E}\left[\left\|\beta_{k}-\beta^{\star}\right\|^{2}_{X^{T}X}\right ]\leq \left(1-\frac{\sigma_{\min}^{2}(X)}{\left\|X\right\|_{F}^{2}}\right)^{k}  \left\|\beta_{0}-\beta^{\star}\right\|_{X^{T}X}^{2}.  \notag\label{Sec3.24}
\end{align}
Later, 
\cite{ma2015convergence} provided a unified analysis of the RK and RGS methods. 
Their convergence performance 
in the above two specific settings are listed in Table \ref{Sec3.2:tab1}.


\begin{table}[t!]
\tblcaption{Summary of convergence properties of the RK and RGS methods for the underdetermined system (\ref{Sec3.11}) and overdetermined system (\ref{Sec3.21}), where $\beta_{LN}$ defined in (\ref{Sec3.13}) denotes the least Euclidean norm solution of (\ref{Sec3.11}) and $\beta^{\star}$ defined in (\ref{Sec3.23}) is the unique solution of (\ref{Sec3.21}).}
{%
\begin{tabular}{@{}ccc@{}}
\tblhead{Method&  Underdetermined system (\ref{Sec3.11}):   &  Overdetermined system (\ref{Sec3.21}):  \cr
 &  convergence to $\beta_{LN}$ ?&    convergence to $\beta^{\star}$ ?}
RK      & \text{Yes} \citep[][]{ma2015convergence}       & \text{Yes} \citep[][]{strohmer2009randomized}      \cr
RGS      & \text{No}  \citep[][]{ma2015convergence}       & \text{Yes} \citep[][]{leventhal2010randomized}
\lastline
\end{tabular}
}
\label{Sec3.2:tab1}
\end{table}

\subsection{SP-RK-RGS method for the ILS problem}
\label{subsec:SP-RK-RGS method}

In the SP method, the update formula (\ref{Sec23}) can be equivalently rewritten as
 \begin{equation}
 A_1^TA_1x_{k+1}=\hat{b},  \label{Sec3.41}
\end{equation}
where $\hat{b}=A_2^TA_2x_{k}+A^TJb$. Considering the characteristics of the RK and RGS methods introduced above, we adopt them 
for solving the following two subsystems of (\ref{Sec3.41}):
\begin{align}
 A_1^Tw&=\hat{b},   \label{Sec3.32111}\\
 A_1z&=w,     \label{Sec3.33111}
\end{align}
in an alternating way. That is, in each iteration, we implement an iteration of the RK method on (\ref{Sec3.32111}) intertwined with an iteration of the RGS method to solve (\ref{Sec3.33111}). 
Thus, we only need to select 
two columns for update in each iteration. 
The specific algorithm is presented in 
\Cref{alg:SP-RK-RGS}.

\begin{algorithm}
\caption{SP-RK-RGS method for the ILS problem (\ref{Sec11}).}
\label{alg:SP-RK-RGS}
\begin{algorithmic}[1]
\STATE{Input: $A$, $J$, $b$, and initial estimate $x_0$. }
\STATE{Set $ \bar{A_2}=A_2^TA_2.$}
\STATE{Set $\bar{b}=A^TJb.$}
\FOR{$k=0, 1, 2, \ldots $ until convergence,}
\STATE{Compute $\hat{b}=\bar{A_2}x_{k}+\bar{b}$.}
\STATE{Set $z_0=w_0=0$.}
\FOR{$t=0, 1, 2, \ldots $ until convergence,}
\STATE{Pick $j_1\in[n]$ with probability $\frac{\left\| A_{1(j_{1})}\right\|_{2}^{2}}{\left\| A_1\right\|_{F}^{2}}$.}
\STATE{Update $w_{t+1}=w_{t}+\frac{ \hat{b}^{(j_{1})}-A_{1(j_{1})}^Tw_t}{ \left\| A_{1\left(j_{1}\right)} \right\|_{2}^{2}}  A_{1(j_{1})}  $.}
\STATE{Pick $j_2\in[n]$ with probability $\frac{\left\| A_{1(j_{2})}\right\|_{2}^{2}}{\left\| A_{1}\right\|_{F}^{2}}$.}
\STATE{Update $z_{t+1}=z_{t}+\frac{A_{1(j_{2})}^T \left(w_{t+1}-A_{1}z_t\right) }{ \left\| A_{1\left(j_{2}\right)} \right\|_{2}^{2}}e_{(j_2)}$.}
\ENDFOR
\STATE{Set $x_{k+1}=z_{t+1}.$ }
\ENDFOR
\end{algorithmic}
\end{algorithm}

\begin{remark}\label{remark-cost: SP-RK-RGS1}
The RK-RGS update in the SP-RK-RGS method is very like the one for the consistent factorized linear
system introduced in \cite{ma2018iterative}. The system is in the following form:
\begin{align}
X\beta=y, \quad   \text{with} \quad  X=UV,  \label{Sec3.31}
\end{align}
where $U \in\mathbb{R}^{t\times r}$, $V \in\mathbb{R}^{r\times l}$ and $\beta$ is the $l$-dimentional unknown vector. The difference lies in that they implemented the RK method twice in each iteration for the two subsystems of (\ref{Sec3.31}):
\begin{align*}
Ux=y, \quad  
 V\beta=x.   
\end{align*}
Thus, for our problem (\ref{Sec3.41}), the RK-RK method has to choose a column and a row of $A_1$ in each iteration. That is, it needs to access both the columns and rows of $A_1$  simultaneously. On the contrary, our RK-RGS update only needs to access the columns of $A_1$.
\end{remark}
\begin{remark}\label{remark-cost: SP-RK-RGS}
In \Cref{alg:SP-RK-RGS}, the main computations of the inner iteration are in step 7 to step 12, which need operation counts of about $\left(2pn+6p+2\right)\cdot T_{\text{RK-RGS}}$, where $T_{\text{RK-RGS}}$ is the number of iterations of the inner RK-RGS update. Updating $x_{k+1}$ from $x_k$ needs to compute step 5 to step 13, which requires operation counts of about $2n^2+\left(2pn+6p+2\right)\cdot T_{\text{RK-RGS}}$. Here, we assume that $T_{\text{RK-RGS}}$ is always the same for $k=0, 1,\ldots$. Then the total operation counts of the SP-RK-RGS method are about
$$qn^2+2mn-n+\left(2n^2+\left(2pn+6p+2\right)\cdot T_{\text{RK-RGS}} \right)\cdot T_{\text{SP-RK-RGS}},$$
where $T_{\text{SP-RK-RGS}}$ is the number of iterations of the outer update of the SP-RK-RGS method. This cost will be much less than the ones of the SP and USSOR methods when $m>p\gg n>q$. Furthermore, like the RK-RK method in \cite{ma2018iterative}, the SP-RK-RGS method can also be accelerated in a parallel computing platform.
\end{remark}

In the following, we consider the convergence of \Cref{alg:SP-RK-RGS}. A preliminary result is first presented as follows.

\begin{theorem} \label{thm:RK-RGS1}
Let $x_{k+1}=( A_1^TA_1)^{-1}\hat{b}$ be the unique solution of (\ref{Sec3.41}), $w^{\star}=A_1(A_1^TA_1)^{-1}\hat{b}$ be the least Euclidean norm solution of (\ref{Sec3.32111}), and $z^{\star}=(A_1^TA_1)^{-1}A_1^Tw^{\star}$ be the unique solution of (\ref{Sec3.33111}). Then solving (\ref{Sec3.32111}) and (\ref{Sec3.33111}) gets the unique solution of (\ref{Sec3.41}), i.e.,
 \begin{align}
 z^{\star}=x_{k+1}.\notag
\end{align}
\end{theorem}

\begin{proof}
The proof is immediate by considering the expressions of $x_{k+1}, w^{\star}$, and $z^{\star}$.
\end{proof}

\begin{theorem} \label{SP-RK-RGS method}
For the ILS problem (\ref{Sec11}), the SP-RK-RGS method, i.e., \Cref{alg:SP-RK-RGS}, converges for any initial vector $x_{0}$.
\end{theorem}

\begin{proof}
Considering  \Cref{thm:split,thm:RK-RGS1} 
, to prove the convergence of the SP-RK-RGS method, 
it suffices to show that the sequence $\left\{z_t\right\}_{t=0}^{\infty}$ generated by the inner iteration, i.e., the RK-RGS update, starting from an initial guess $z_0=0$, converges to $z^{\star}$ in expectation.

To the above end, we first set $\tilde{z}_{t}=z_{t-1}+\frac{A_{1(j_{2})}^T \left(w^{\star}-A_{1}z_{t-1}\right) }{ \| A_{1\left(j_{2}\right)} \|_{2}^{2}}e_{(j_2)}$. Then
 \begin{align}
\mathbb{E}^{t-1}\left[\left\| z_t-z^{\star}\right\|^2_{A_{1}^TA_{1}}\right]
 &=\mathbb{E}^{t-1}\left[\left\| A_{1}z_t-A_{1}z^{\star}\right\|^2_{2}\right] \notag
 \\
&=\mathbb{E}^{t-1}\left[\left\| A_{1}z_t-A_{1}z^{\star}+A_{1}\tilde{z}_{t}-A_{1}\tilde{z}_{t}\right\|^2_{2}\right] \notag
 \\
&=\mathbb{E}^{t-1}\left[\left\|  A_{1}\tilde{z}_{t}-A_{1}z^{\star}\right\|^2_{2}\right] +\mathbb{E}^{t-1}\left[\left\| A_{1}z_t -A_{1}\tilde{z}_{t}\right\|^2_{2}\right] \notag
 \\
&+2\mathbb{E}^{t-1}\left[\left < A_{1}\tilde{z}_{t}-A_{1}z^{\star}, A_{1}z_t -A_{1}\tilde{z}_{t}\right> \right].\label{Sec3.34}
\end{align}

Next, we show that $\mathbb{E}^{t-1}\left[\left < A_{1}\tilde{z}_{t}-A_{1}z^{\star}, A_{1}z_t -A_{1}\tilde{z}_{t}\right> \right]=0$. From \Cref{alg:SP-RK-RGS} and the definition of $\tilde{z}_{t}$, it follows that
 \begin{align}
&\mathbb{E}^{t-1}\left[\left < A_{1}\tilde{z}_{t}-A_{1}z^{\star}, A_{1}z_t -A_{1}\tilde{z}_{t}\right> \right] \notag
\\
&= \mathbb{E}^{t-1}\left[ \left <A_{1}z_{t-1}-A_{1}z^{\star}+ \frac{A_{1(j_{2})}^T \left(w^{\star}-A_{1}z_{t-1}\right) }{ \| A_{1\left(j_{2}\right)} \|_{2}^{2}}A_{1(j_2)},\frac{A_{1(j_{2})}^T \left(w_{t }-w^{\star}\right) }{ \| A_{1\left(j_{2}\right)} \|_{2}^{2}}A_{1(j_2)}   \right>  \right] \notag
\\
&=\mathbb{E}^{t-1}\left[ \left <A_{1}z_{t-1}-A_{1}z^{\star},\frac{A_{1(j_{2})}^T \left(w_{t }-w^{\star}\right) }{ \| A_{1\left(j_{2}\right)} \|_{2}^{2}}A_{1(j_2)}  \right>  \right]\notag
\\
&\quad+\mathbb{E}^{t-1}\left[ \left < \frac{A_{1(j_{2})}^T \left(w^{\star}-A_{1}z_{t-1}\right) }{ \| A_{1\left(j_{2}\right)} \|_{2}^{2}}A_{1(j_2)},\frac{A_{1(j_{2})}^T \left(w_{t }-w^{\star}\right) }{ \| A_{1\left(j_{2}\right)} \|_{2}^{2}}A_{1(j_2)} \right>  \right]\notag
\\
&=\sum_{j_{2}=1}^{n}\frac{\left \| A_{1(j_{2})} \right\|_{2}^{2}}{\left \| A_{1} \right\|_{F}^{2}}\left <A_{1}z_{t-1}-A_{1}z^{\star},\frac{A_{1(j_{2})}^T \left(w_{t }-w^{\star}\right) }{ \| A_{1\left(j_{2}\right)} \|_{2}^{2}}A_{1(j_2)}  \right>\notag
\\
&\quad+\sum_{j_{2}=1}^{n}\frac{\left \| A_{1(j_{2})}\right\|_{2}^{2}}{\left \| A_{1}\right\|_{F}^{2}}  \frac{A_{1(j_{2})}^T \left(w^{\star}-A_{1}z_{t-1}\right)\cdot A_{1(j_{2})}^T \left(w_{t }-w^{\star}\right) }{ \| A_{1\left(j_{2}\right)} \|_{2}^{2}}   \notag
\\
&=\frac{\left< A_{1}z_{t-1}-A_{1}z^{\star}, A_{1}A_{1}^T\left(w_{t }-w^{\star}\right)\right>}{\left \| A_{1} \right\|_{F}^{2}}+\frac{\left< A_{1}^T\left(w^{\star}-A_{1}z_{t-1}\right), A_{1}^T \left(w_{t }-w^{\star}\right)\right>}{\left \| A_{1} \right\|_{F}^{2}}, \notag
\end{align}
which together with the fact that $A_{1}z^{\star}=w^{\star}$ yields
\begin{align}
\mathbb{E}^{t-1}\left[\left < A_{1}\tilde{z}_{t}-A_{1}z^{\star}, A_{1}z_t -A_{1}\tilde{z}_{t}\right> \right] &=\frac{\left< A_{1}z_{t-1}-w^{\star}, A_{1}A_{1}^T\left(w_{t }-w^{\star}\right)\right>}{\left \| A_{1} \right\|_{F}^{2}}\notag
\\
&+\frac{\left< A_{1}^T\left(w^{\star}-A_{1}z_{t-1}\right), A_{1}^T \left(w_{t }-w^{\star}\right)\right>}{\left \| A_{1} \right\|_{F}^{2}}=0.\notag
\end{align}
So, the desired result holds. Therefore, (\ref{Sec3.34}) is reduced to
 \begin{align}
 \mathbb{E}^{t-1}\left[\left\| z_t-z^{\star}\right\|^2_{A_{1}^TA_{1}}\right]=\mathbb{E}^{t-1}\left[\left\|  A_{1}\tilde{z}_{t}-A_{1}z^{\star}\right\|^2_{2}\right] +\mathbb{E}^{t-1}\left[\left\| A_{1}z_t -A_{1}\tilde{z}_{t}\right\|^2_{2}\right] .\label{Sec3.35}
\end{align}

Next, we show that $\left\|  A_{1}\tilde{z}_{t}-A_{1}z^{\star}\right\|^2_{2}=\left\|  A_{1}z_{t-1}-A_{1}z^{\star}\right\|^2_{2}-\left\|  A_{1}\tilde{z}_{t}-A_{1}z_{t-1}\right\|^2_{2}$. From the update formula of $\tilde{z}_{t}$, we have
\begin{align}
A_{1}\left(\tilde{z}_{t}- z_{t-1}\right)= \frac{A_{1(j_{2})}^T \left(w^{\star}-A_{1}z_{t-1}\right) }{ \| A_{1\left(j_{2}\right)} \|_{2}^{2}}A_{1(j_2)}, \label{Sec3.35.1}
\end{align}
which implies that $ A_{1}\left(\tilde{z}_{t}- z_{t-1}\right)$ is parallel to $A_{1\left(j_{2}\right)}$. Meanwhile,
\begin{align}
A_{1}\left(\tilde{z}_{t}- z^{\star}\right)&=A_{1}\left(z_{t-1}-z^{\star}+\frac{A_{1(j_{2})}^T \left(w^{\star}-A_{1}z_{t-1}\right) }{ \left\| A_{1\left(j_{2}\right)} \right\|_{2}^{2}}e_{(j_2)}\right),\notag
\end{align}
which together with the fact that $A_{1}z^{\star}=w^{\star}$ gives
\begin{align}
A_{1}\left(\tilde{z}_{t}- z^{\star}\right)&=\left(I- \frac{ A_{1\left(j_{2}\right)}  A_{1\left(j_{2}\right)} ^T}{\left\| A_{1\left(j_{2}\right)} \right\|_{2}^{2}}\right)A_{1}\left( z_{t-1}-z_\star\right).\notag
\end{align}
Further, we can check
\begin{align}
A_{1(j_{2})}^T A_{1}\left(\tilde{z}_{t}- z^{\star}\right)&=A_{1(j_{2})}^T \left(I- \frac{ A_{1\left(j_{2}\right)}  A_{1\left(j_{2}\right)} ^T}{\left\| A_{1\left(j_{2}\right)} \right\|_{2}^{2}}\right)A_{1}\left( z_{t-1}-z_\star\right)=0.\notag
\end{align}
Then $A_{1}\left(\tilde{z}_{t}- z^{\star}\right)$ is orthogonal to $A_{1(j_{2})}$ and hence the vector $A_{1}\left(\tilde{z}_{t}- z_{t-1}\right)$ is perpendicular to the vector $A_{1}\left(\tilde{z}_{t}- z^{\star}\right)$. Thus, by the Pythagorean theorem, we get the desired result
$$\left\|  A_{1}\tilde{z}_{t}-A_{1}z^{\star}\right\|^2_{2}=\left\|  A_{1}z_{t-1}-A_{1}z^{\star}\right\|^2_{2}-\left\|  A_{1}\tilde{z}_{t}-A_{1}z_{t-1}\right\|^2_{2}.$$
Substituting it into (\ref{Sec3.35}) leads to
 \begin{align}
 \mathbb{E}^{t-1}\left[\left\| z_t-z^{\star}\right\|^2_{A_{1}^TA_{1}}\right]
 &=\mathbb{E}^{t-1}\left[\left\|  A_{1}z_{t-1}-A_{1}z^{\star}\right\|^2_{2}\right]-\mathbb{E}^{t-1}\left[\left\|  A_{1}\tilde{z}_{t}-A_{1}z_{t-1}\right\|^2_{2}\right] \notag
  \\
  &\quad+\mathbb{E}^{t-1}\left[\left\| A_{1}z_t -A_{1}\tilde{z}_{t}\right\|^2_{2}\right] .\notag
\end{align}
Thus, by using (\ref{Sec3.35.1}), the update rule of $z_{t}$, and $\mathbb{E}^{t-1}=\mathbb{E}^{t-1}_{w}\mathbb{E}^{t-1}_{z}$, we have
 \begin{align}
 \mathbb{E}^{t-1}\left[\left\| z_t-z^{\star}\right\|^2_{A_{1}^TA_{1}}\right]
 &= \left\|  A_{1}z_{t-1}-A_{1}z^{\star}\right\|^2_{2} -\mathbb{E}^{t-1}\left[\left\|  \frac{A_{1(j_{2})}^T \left(w^{\star}-A_{1}z_{t-1}\right) }{ \left\| A_{1\left(j_{2}\right)} \right\|_{2}^{2}}A_{1(j_2)}\right\|^2_{2}\right] \notag
  \\
  &\quad+\mathbb{E}^{t-1}\left[\left\| \frac{A_{1(j_{2})}^T \left(w_{t }-w^{\star}\right) }{ \| A_{1\left(j_{2}\right)} \|_{2}^{2}}A_{1(j_2)} \right\|^2_{2}\right]  \notag
  \\
  &= \left\|  A_{1}z_{t-1}-A_{1}z^{\star}\right\|^2_{2} -\sum_{j_{2}=1}^{n}\frac{\left \| A_{1(j_{2})}\right\|_{2}^{2}}{\left \| A_{1}\right\|_{F}^{2}}  \frac{\left(A_{1(j_{2})}^T \left(w^{\star}-A_{1}z_{t-1}\right) \right)^2}{ \left\| A_{1\left(j_{2}\right)} \right\|_{2}^{2}} \notag
  \\
  &\quad+\mathbb{E}^{t-1}_{w}\mathbb{E}^{t-1}_{z}\left[ \frac{\left(A_{1(j_{2})}^T \left(w_{t }-w^{\star}\right)\right)^2 }{ \left\| A_{1\left(j_{2}\right)} \right\|_{2}^{2}} \right]  \notag
 \\
  &= \left\|  A_{1}z_{t-1}-A_{1}z^{\star}\right\|^2_{2} - \frac{ \left \| A_{1}^T \left(w^{\star}-A_{1}z_{t-1}\right) \right\|_2 ^2}{\left \| A_{1}\right\|_{F}^{2}} +\mathbb{E}^{t-1}_{w} \left[ \frac{\left\| A_{1} ^T \left(w_{t }-w^{\star}\right)\right\|^2_{2} }{ \left\| A_{1}\right\|_{F}^{2}} \right]. \notag
\end{align}
Further, noting $A_{1}z^{\star}=w^{\star}$ and $\left\|A_{1}^T\left(  A_{1}z_{t-1}-A_{1}z^{\star}\right)\right\|^2_{2}\geq\sigma_{\min}^{2}(A_{1})\left\|  A_{1}z_{t-1}-A_{1}z^{\star} \right\|^2_{2}$, we get
 \begin{align}
 \mathbb{E}^{t-1}\left[\left\| z_t-z^{\star}\right\|^2_{A_{1}^TA_{1}}\right]
  & \leq  \left(1-\frac{\sigma_{\min}^{2}(A_{1})}{\left\|A_{1}\right\|_{F}^{2}}\right) \left\|  A_{1}z_{t-1}-A_{1}z^{\star}\right\|^2_{2} +\mathbb{E}^{t-1}_{w} \left[ \frac{\left\| A_{1} ^T \left(w_{t }-w^{\star}\right)\right\|^2_{2} }{ \left\| A_{1}\right\|_{F}^{2}} \right]  \notag
  \\
   & \leq  \left(1-\frac{\sigma_{\min}^{2}(A_{1})}{\left\|A_{1}\right\|_{F}^{2}}\right) \left\|   z_{t-1}- z^{\star}\right\|^2_{A_{1}^TA_{1}} +\frac{\sigma_{\max}^{2}(A_{1}) }{ \left\| A_{1}\right\|_{F}^{2}}  \mathbb{E}^{t-1}_{w} \left[\left\| w_{t }-w^{\star} \right\|^2_{2}  \right],  \notag
\end{align}
which together with a result derived from the convergence property of the RK method discussed in (\ref{Sec3.14}), i.e.,
 \begin{align}
\mathbb{E} \left[\left\| w_{t }-w^{\star} \right\|^2_{2} \right]  \leq \left(1-\frac{\sigma_{\min}^{2}(A_{1})}{\|A_{1}\|_{F}^{2}}\right)^{t}  \left\| w^{\star}\right\|_{2}^{2}, \notag
\end{align}
and the law of total expectation, implies
 \begin{align}
  &\mathbb{E} \left[\left\| z_t-z^{\star}\right\|^2_{A_{1}^TA_{1}}\right]\notag
  \\
   & \leq   \left(1-\frac{\sigma_{\min}^{2}(A_{1})}{\left\|A_{1}\right\|_{F}^{2}}\right) \mathbb{E} \left[ \left\|   z_{t-1}- z^{\star}\right\|^2_{A_{1}^TA_{1}} \right]+\frac{\sigma_{\max}^{2}(A_{1}) }{ \left\| A_{1}\right\|_{F}^{2}}  \left(1-\frac{\sigma_{\min}^{2}(A_{1})}{\left\|A_{1}\right\|_{F}^{2}}\right)^{t}  \left\| w^{\star}\right\|_{2}^{2}   \notag
   \\
   & \leq   \left(1-\frac{\sigma_{\min}^{2}(A_{1})}{\left\|A_{1}\right\|_{F}^{2}}\right)^2 \mathbb{E} \left[ \left\|   z_{t-2}- z^{\star}\right\|^2_{A_{1}^TA_{1}} \right]+2\frac{\sigma_{\max}^{2}(A_{1}) }{ \left\| A_{1}\right\|_{F}^{2}}  \left(1-\frac{\sigma_{\min}^{2}(A_{1})}{\left\|A_{1}\right\|_{F}^{2}}\right)^{t}  \left\| w^{\star}\right\|_{2}^{2}   \notag
    \\
   & \leq \ldots  \leq \left(1-\frac{\sigma_{\min}^{2}(A_{1})}{\left\|A_{1}\right\|_{F}^{2}}\right)^t \left\|   z^{\star}\right\|^2_{A_{1}^TA_{1}}  +t\frac{\sigma_{\max}^{2}(A_{1}) }{ \left\| A_{1}\right\|_{F}^{2}}  \left(1-\frac{\sigma_{\min}^{2}(A_{1})}{\left\|A_{1}\right\|_{F}^{2}}\right)^{t}  \left\| w^{\star}\right\|_{2}^{2}.  \notag
\end{align}
This completes the proof.
 \end{proof}

\begin{remark}
\label{re-th:sp-rk-rgs}
Similar to the discussion in \Cref{re-sp} for the SP method, from \Cref{SP-RK-RGS method}, we find that the SP-RK-RGS method also converges `unconditionally'.
\end{remark}

 \section{SP-SCD method for the ILS problem}
\label{sec:randomized sampling method}

In \Cref{alg:SP-RK-RGS}, if we set the two columns in each iteration to be the same, i.e.,  $j_1=j_2=j$, then the inner iteration, i.e., the RK-RGS update, reduces to the RCD update. Specifically,
 \begin{align}
z_{t+1}
&=
z_{t}+\frac{A_{1(j)}^T \left(w_{t+1}-A_{1}z_t\right) }{ \left\| A_{1\left(j\right)} \right\|_{2}^{2}}e_{(j)}  \notag
\\
&=z_{t}+\frac{A_{1(j)}^T \left(w_{t}+\frac{ \hat{b}^{(j)}-A_{1(j)}^Tw_t}{ \left\| A_{1\left(j\right)} \right\|_{2}^{2}}  A_{1(j)}-A_{1}z_t\right) }{ \left\| A_{1\left(j\right)} \right\|_{2}^{2}}e_{(j)}
\notag
\\
&=z_{t}+\frac{A_{1(j)}^T w_{t}+  \hat{b}^{(j)}-A_{1(j)}^Tw_t -A_{1(j)}^TA_{1}z_t  }{ \left\| A_{1\left(j\right)} \right\|_{2}^{2}}e_{(j)}
\notag
\\
&=z_{t}+\frac{   \hat{b}^{(j)}  -A_{1(j)}^TA_{1}z_t  }{ \left(A_{1}^TA_{1}\right)_{\left(j, j\right)}}e_{(j)},
\notag
\end{align}
where $ \hat{b}^{(j)}  -A_{1(j)}^TA_{1}z_t$ is the $j$-th coordinate of the gradient and $ \left(A_{1}^TA_{1}\right)_{\left(j, j\right)}$ is its Lipschitz constant. Hence, the above formula can be seen as the CD update for $\min\limits_{z}\frac{1}{2}z^TA_{1}^TA_{1}z-\hat{b}^Tz$, which has the same solution as the positive definite linear system (\ref{Sec3.41}) \citep[][]{leventhal2010randomized}. By the way, the relationship between the RK, RGS and RCD methods was discussed in \citet{hefny2017rows} in detail. In particular, the RK and RGS methods can be viewed as different variants of the RCD method.

Based on the above discussions and inspired by \citet{de2017sampling} and \citet{haddock2021greed}, similar to the SP-RK-RGS method, we  propose the SP-SCD method for solving the ILS problem (\ref{Sec11}). That is, the inner iteration in the SP-RK-RGS method is replaced by the sampling coordinate descent (SCD) update. The specific algorithm is summarized in \Cref{alg:SP-SCD}.
\begin{algorithm}
\caption{SP-SCD method for the ILS problem (\ref{Sec11}).}
\label{alg:SP-SCD}
\begin{algorithmic}[1]
\STATE{Input: $A$, $J$, $b$, and initial estimate $x_0$. }
\STATE{Set $ \bar{A_1}=A_1^TA_1.$}
\STATE{Set $ \bar{A_2}=A_2^TA_2.$}
\STATE{Set $\bar{b}=A^TJb.$}
\FOR{$k=0, 1, 2, \ldots $ until convergence,}
\STATE{Compute $\hat{b}=\bar{A_2}x_{k}+\bar{b}$.}
\STATE{Set $\beta_0=0$.}
\FOR{$t=0, 1, 2, \ldots $ until convergence,}
\STATE{Generate a positive integer $\alpha_t\in[n]$ at random.}
\STATE{Choose an index subset $\tau_t$ of size $\alpha_t$  from among $ [n]$ with probability
 \begin{align}
p\left(\tau_t, \beta_{t}\right)=\frac{ \bar{A_1}_{\left(s\left(\tau_{t}, \beta_{t}\right ), s\left(\tau_{t}, \beta_{t}\right )\right)} }{\sum\limits_{\tau\in\binom{[n]}{\alpha_t}} \bar{A_1}_{\left(s\left(\tau, \beta_{t}\right ), s\left(\tau, \beta_{t}\right )\right)} },\label{Sec4.111}
\end{align}
where $s(\tau, \beta_{t})=\text{arg}\max\limits_{s\in\tau}\left(\hat{b}^{(s)}-\bar{A_1}^{(s)}\beta_{t}\right)^2$.}
\STATE{Set $j_{t}=s(\tau_t, \beta_{t})$.}
\STATE{Update $\beta_{t+1}=\beta_{t}+\frac{ \hat{b}^{\left(j_{t}\right)}-\bar{A_1}^{\left(j_{t}\right)}\beta_t}{ \bar{A_1}_{\left(j_{t}, j_{t}\right)}}  e_{(j_{t})}  $.}
\ENDFOR
\STATE{Set $x_{k+1}=\beta_{t+1}.$ }
\ENDFOR
\end{algorithmic}
\end{algorithm}

\begin{remark}\label{remark-cost: SP-SCD}
Unlike the SP-RK-RGS method or the SP-RCD method (it is the immediate result of the SP-RK-RGS method with $j_1=j_2=j$),
the probability utilized in the SP-SCD method is adaptive. Specifically, the probability used in \Cref{alg:SP-SCD}, i.e., (\ref{Sec4.111}),  depends on the value of $\bar{A_1}_{\left(s\left(\tau_{t}, \beta_{t}\right ), s\left(\tau_{t}, \beta_{t}\right )\right)} $, which gives the largest residual value among $\left(\hat{b}^{(s)}-\bar{A_1}^{(s)}\beta_{t}\right)^2$ where $ s\in\tau_t $. In particular, if $\alpha_t=1$, the probability reduces to
\begin{align}
p\left(\tau_t, \beta_{t}\right)=\frac{ \bar{A_1}_{\left(j_{t}, j_{t}\right )} }{\sum\limits_{j_{t}\in [n]}  \bar{A_1}_{\left(j_{t}, j_{t}\right )} }=\frac{\| A_{1(j_{k})}\|_{2}^{2}}{\| A_1\|_{F}^{2}},\notag
\end{align}
which is a fixed probability 
equivalent to the one of the RGS method listed in (\ref{Sec3.21.1}). If $\alpha_t=n$, the probability reduces to
\begin{align}
p\left(\tau_t, \beta_{t}\right)=1,\notag
\end{align}
which is equivalent to grasping the index corresponding to the largest magnitude entry of the residual vector as used in the Motzkin method \citep[][]{motzkin_schoenberg_1954}. If $\bar{A_1}_{\left(i, i\right )}=\bar{A_1}_{\left(j, j\right )}$ for any $i, j \in [n]$, the probability reduces to
\begin{align}
p\left(\tau_t, \beta_{t}\right)=\frac{1}{\binom{[n]}{\alpha_t}},\notag
\end{align}
which is a uniform probability 
equivalent to the strategy discussed in \cite{de2017sampling}. In the numerical experiments in \Cref{sec:experiments}, we mainly consider the last strategy. 
\end{remark}

Now, we present the convergence analysis for the SP-SCD method.
\begin{theorem} \label{SP-SCD method}
For the ILS problem (\ref{Sec11}), the SP-SCD method, i.e., \Cref{alg:SP-SCD}, converges for any initial vector $x_{0}$.
\end{theorem}

\begin{proof}
 Considering \Cref{thm:split} and 
 the assumption $\beta^{\star}=x_{k+1}$, where $\beta^{\star}$ is the unique solution of the rewritten form of (\ref{Sec3.41}), i.e., $\bar{A_1}\beta=\hat{b}$, to prove the convergence of the SP-SCD method, we only need to show that the sequence $\left\{\beta_t\right\}_{t=0}^{\infty}$ generated by the inner iteration, i.e., the SCD update, starting from an initial guess $\beta_0=0$, converges to $\beta^{\star}$ in expectation.

First, from \Cref{alg:SP-SCD}, we have
 \begin{align}
\beta_{t}-\beta^{\star}
&=
\beta_{t-1}-\beta^{\star}+\frac{ \hat{b}^{\left(j_{t-1}\right)}-\bar{A_1}^{\left(j_{t-1}\right)}\beta_{t-1}}{ \bar{A}_{1\left(j_{t-1}, j_{t-1}\right)}}  e_{(j_{t-1})}
=
\beta_{t-1}-\beta^{\star}-\frac{e_{\left(j_{t-1}\right)}^T\left( \bar{A_1}\beta_{t-1}-\hat{b} \right)}{ \bar{A}_{1\left(j_{t-1}, j_{t-1}\right)}}  e_{(j_{t-1})},
\notag
\end{align}
which together with the fact $\bar{A_1}\beta^{\star}=\hat{b}$ yields
 \begin{align}
\beta_{t}-\beta^{\star}
&=
\left(I- \frac{e_{(j_{t-1})}e_{(j_{t-1})}^T\bar{A_1}}{ \bar{A}_{1\left(j_{t-1}, j_{t-1}\right)}}\right)  \left( \beta_{t-1}-\beta^{\star} \right).
\notag
\end{align}
Thus, taking the square of the energy norm on both sides, by some algebra, we get
 \begin{align}
\left\| \beta_{t}-\beta^{\star}  \right\|_{\bar{A_1}}^{2}
&=
\left( \beta_{t}-\beta^{\star} \right)^T\bar{A_1}\left( \beta_{t}-\beta^{\star} \right)
\notag
\\
&= \left( \beta_{t-1}-\beta^{\star} \right)^T\left(I- \frac{\bar{A_1} e_{(j_{t-1})} e_{(j_{t-1})}^T}{ \bar{A}_{1\left(j_{t-1}, j_{t-1}\right)}}\right)\bar{A_1}\left(I- \frac{e_{(j_{t-1})}e_{(j_{t-1})}^T\bar{A_1}}{ \bar{A}_{1\left(j_{t-1}, j_{t-1}\right)}}\right)  \left( \beta_{t-1}-\beta^{\star} \right)
\notag
\\
&= \left( \beta_{t-1}-\beta^{\star} \right)^T\left(\bar{A_1}- \frac{\bar{A_1} e_{(j_{t-1})} e_{(j_{t-1})}^T\bar{A_1}}{ \bar{A}_{1\left(j_{t-1}, j_{t-1}\right)}} \right)  \left( \beta_{t-1}-\beta^{\star} \right)
\notag
\\
&= \left\| \beta_{t-1}-\beta^{\star}  \right\|_{\bar{A_1}}^{2} -  \frac{\left( \beta_{t-1}-\beta^{\star} \right)^T\bar{A_1}^T e_{(j_{t-1})} e_{(j_{t-1})}^T\bar{A_1}\left( \beta_{t-1}-\beta^{\star} \right)}{ \bar{A}_{1\left(j_{t-1}, j_{t-1}\right)}}
\notag
\\
&= \left\| \beta_{t-1}-\beta^{\star}  \right\|_{\bar{A_1}}^{2} -  \frac{\left(e_{(j_{t-1})}^T\bar{A_1}\left( \beta_{t-1}-\beta^{\star} \right)\right)^2}{ \bar{A}_{1\left(j_{t-1}, j_{t-1}\right)}}
\notag
\\
&= \left\| \beta_{t-1}-\beta^{\star}  \right\|_{\bar{A_1}}^{2} -  \frac{ \left(\bar{A_1}^{(j_{t-1})}\beta_{t-1}-\hat{b}^{(j_{t-1})}\right)^2 }{ \bar{A}_{1\left(j_{t-1}, j_{t-1}\right)}}.
\notag
\end{align}
Now, taking expectation of both sides (with respect to $\tau_{t-1}$) conditioned on $\beta_{t-1}$, we obtain
\begin{align}
&\mathbb{E}^{t-1}_{\tau_{t-1}}\left[ \left\| \beta_{t}-\beta^{\star}  \right\|_{\bar{A_1}}^{2}  \right]
=
\left\| \beta_{t-1}-\beta^{\star}  \right\|_{\bar{A_1}}^{2} -  \mathbb{E}^{t-1}_{\tau_{t-1}}\left[\frac{ \left(\bar{A_1}^{(j_{t-1})}\beta_{t-1}-\hat{b}^{(j_{t-1})}\right)^2 }{ \bar{A}_{1\left(j_{t-1}, j_{t-1}\right)}}\right]
\notag
\\
&\quad\quad=
\left\| \beta_{t-1}-\beta^{\star}  \right\|_{\bar{A_1}}^{2} -  \sum\limits_{\tau\in\binom{[n]}{\alpha_{t-1}}} p\left(\tau, \beta_{t-1}\right )\cdot\frac{ \left(\bar{A_1}^{(j_{t-1})}\beta_{t-1}-\hat{b}^{(j_{t-1})}\right)^2  }{ \bar{A}_{1\left(j_{t-1}, j_{t-1}\right)}}
\notag
\\
&\quad\quad=
\left\| \beta_{t-1}-\beta^{\star}  \right\|_{\bar{A_1}}^{2} -  \sum\limits_{\tau\in\binom{[n]}{\alpha_{t-1}}} \frac{ \bar{A}_{1\left(s\left(\tau, \beta_{t-1}\right ), s\left(\tau, \beta_{t-1}\right )\right)} }{\sum\limits_{\upsilon\in\binom{[n]}{\alpha_{t-1}}} \bar{A}_{1\left(s\left(\upsilon, \beta_{t-1}\right ), s\left(\upsilon, \beta_{t-1}\right )\right)} } \cdot\frac{ \left\|\bar{A_1}^{\tau}\beta_{t-1}-\hat{b}^{\tau}\right\|^2_{\infty} }{ \bar{A}_{1\left(s\left(\tau, \beta_{t-1}\right ), s\left(\tau, \beta_{t-1}\right )\right)} }
\notag
\\
&\quad\quad=
\left\| \beta_{t-1}-\beta^{\star}  \right\|_{\bar{A_1}}^{2} - \frac{ 1 }{\sum\limits_{\upsilon\in\binom{[n]}{\alpha_{t-1}}} \bar{A}_{1\left(s\left(\upsilon, \beta_{t-1}\right ), s\left(\upsilon, \beta_{t-1}\right )\right)} } \sum\limits_{\tau\in\binom{[n]}{\alpha_{t-1}}}\left\| \bar{A_1}^{\tau}\beta_{t-1}-\hat{b}^{\tau}\right\|^2_{\infty},
\notag
\end{align}
which together with
\begin{align}
\xi_{j}=\frac{\sum\limits_{\tau\in\binom{[n]}{\alpha_{j}}}\left\| \bar{A_1}^{\tau}\beta_{j}-\hat{b}^{\tau}\right\|^2_{2}}{\sum\limits_{\tau\in\binom{[n]}{\alpha_{j}}}\left\| \bar{A_1}^{\tau}\beta_{j}-\hat{b}^{\tau}\right\|^2_{\infty}},\notag
\end{align}
leads to
\begin{align}
&\mathbb{E}^{t-1}_{\tau_{t-1}}\left[ \left\| \beta_{t}-\beta^{\star}  \right\|_{\bar{A_1}}^{2}  \right]\notag
\\
&
=
\left\| \beta_{t-1}-\beta^{\star}  \right\|_{\bar{A_1}}^{2} - \frac{ 1 }{\sum\limits_{\upsilon\in\binom{[n]}{\alpha_{t-1}}} \bar{A}_{1\left(s\left(\upsilon, \beta_{t-1}\right ), s\left(\upsilon, \beta_{t-1}\right )\right)} } \cdot \frac{1}{\xi_{t-1}}\cdot \sum\limits_{\tau\in\binom{[n]}{\alpha_{t-1}}}\left\| \bar{A_1}^{\tau}\beta_{t-1}-\hat{b}^{\tau}\right\|^2_{2}
\notag
\\
&=
\left\| \beta_{t-1}-\beta^{\star}  \right\|_{\bar{A_1}}^{2} - \frac{ 1 }{\sum\limits_{\upsilon\in\binom{[n]}{\alpha_{t-1}}} \bar{A}_{1\left(s\left(\upsilon, \beta_{t-1}\right ), s\left(\upsilon, \beta_{t-1}\right )\right)} } \cdot \frac{1}{\xi_{t-1}}\cdot \frac{\binom{n}{\alpha_{t-1}}\alpha_{t-1}}{n}\cdot \left\| \bar{A_1}\beta_{t-1}-\hat{b}\right\|^2_{2}.
\notag
\end{align}
Further, noting $\bar{A_1}\beta^{\star}=\hat{b}$ and $\bar{A_1}=A_1^TA_1$, we have
\begin{align}
&\mathbb{E}^{t-1}_{\tau_{t-1}}\left[ \left\| \beta_{t}-\beta^{\star}  \right\|_{\bar{A_1}}^{2}  \right]
\notag
\\
&=
\left\| \beta_{t-1}-\beta^{\star}  \right\|_{\bar{A_1}}^{2} - \frac{ 1 }{\sum\limits_{\upsilon\in\binom{[n]}{\alpha_{t-1}}} \bar{A}_{1\left(s\left(\upsilon, \beta_{t-1}\right ), s\left(\upsilon, \beta_{t-1}\right )\right)} } \cdot \frac{1}{\xi_{t-1}}\cdot \frac{\binom{n}{\alpha_{t-1}}\alpha_{t-1}}{n}\cdot \left\| A_1^TA_1\left(\beta_{t-1}-\beta^{\star}\right)\right\|^2_{2}
\notag
\\
&\leq
\left\| \beta_{t-1}-\beta^{\star}  \right\|_{\bar{A_1}}^{2}- \frac{ 1 }{\sum\limits_{\upsilon\in\binom{[n]}{\alpha_{t-1}}} \bar{A}_{1\left(s\left(\upsilon, \beta_{t-1}\right ), s\left(\upsilon, \beta_{t-1}\right )\right)} } \cdot \frac{1}{\xi_{t-1}}\cdot \frac{\binom{n}{\alpha_{t-1}}\alpha_{t-1}}{n}\cdot \sigma_{\min}^{2}(A_1) \left\| A_1\left(\beta_{t-1}-\beta^{\star}\right)  \right\|_{2}^{2}
\notag
\\
&\leq
\left( 1 - \frac{ 1 }{\sum\limits_{\upsilon\in\binom{[n]}{\alpha_{t-1}}} \bar{A}_{1\left(s\left(\upsilon, \beta_{t-1}\right ), s\left(\upsilon, \beta_{t-1}\right )\right)} } \cdot \frac{1}{\xi_{t-1}}\cdot \frac{\binom{n}{\alpha_{t-1}}\alpha_{t-1}}{n}\cdot \sigma_{\min}^{2}(A_1)\right)\left\| \beta_{t-1}-\beta^{\star}  \right\|_{\bar{A_1}}^{2}.
\notag
\end{align}
Thus, by the law of total expectation, we can obtain
\begin{align}
\mathbb{E}\left[ \left\| \beta_{t}-\beta^{\star}  \right\|_{\bar{A_1}}^{2}  \right]\leq
\prod\limits_{j=0}^{t-1}\left( 1 - \frac{ 1 }{\sum\limits_{\upsilon\in\binom{[n]}{\alpha_{j}}} \bar{A}_{1\left(s\left(\upsilon, \beta_{j}\right ), s\left(\upsilon, \beta_{j}\right )\right)} } \cdot \frac{1}{\xi_{j}}\cdot \frac{\binom{n}{\alpha_{j}}\alpha_{j}}{n}\cdot \sigma_{\min}^{2}(A_1)\right)\left\| \beta^{\star}  \right\|_{\bar{A_1}}^{2},
\notag
\end{align}
which concludes the proof.
 \end{proof}

\begin{remark}
\label{re-th:sp-scd}
Similar to the SP and SP-RK-RGS methods, 
from \Cref{SP-SCD method}, it follows that the SP-SCD method also converges `unconditionally'.
\end{remark}
\section{Experimental results}
\label{sec:experiments}

In this section, we compare the latest iterative method for the ILS problem, i.e., the USSOR method, with our proposed methods, i.e., the SP, SP-RK-RGS, and SP-SCD methods, in terms of the computing time in seconds (denoted as ``CPU") and  the number of iterations (denoted as ``IT"). Here, the  CPU and IT are arithmetical average quantities with respect to 10 repeated trials of each method. We also use CPU-inner and IT-inner to represent respectively the total inner computing time and iteration numbers of the inner iterations of the SP-RK-RGS and SP-SCD methods. Furthermore, to see the advantage of our proposed SP, SP-RK-RGS and SP-SCD methods over the USSOR method more intuitively, we also present the computing time speed-up of our methods against the USSOR method, which are defined as
\begin{align}
\texttt{speed-up-1}=\frac{\text{CPU of USSOR} }{\text{CPU of SP} },\quad
\texttt{speed-up-2}=\frac{\text{CPU of USSOR} }{\text{CPU of SP-RK-RGS} },\notag
\end{align}
and
\begin{align}
\texttt{speed-up-3}=\frac{\text{CPU of USSOR} }{\text{CPU of SP-SCD} }.\notag
\end{align}
All the computations are obtained by using MATLAB (version R2017a) on a personal computer with 3.00 GHz CPU (Intel(R) Core(TM) i7-9700), 16.0 GB memory, and Windows 10 operating system.

In addition, all the  experiments start from an initial vector $x_0=0$, and terminate once the \emph{relative residual} (RR) at $x_{k}$, defined by
$$\mathrm{RR}=\frac{\left\|A^TJ\left(Ax_k-b \right)\right\|_2^2 }{\left\|A^TJb\right\|_2^2 },$$
is less than $10^{-6}$, or the number of outer iterations exceeds 20000.

\subsection{Computational complexities}
\label{subsec:Computational complexities}
Before showing the specific experimental results, we first discuss the computational complexities of the USSOR method listed in \Cref{alg:USSOR} and our proposed methods, i.e., the SP, SP-RK-RGS and SP-SCD methods.
\begin{algorithm}
\caption{USSOR method for the ILS problem (\ref{Sec11}) \citep[][]{song2020ussor}.}
\label{alg:USSOR}
\begin{algorithmic}[1]
\STATE{ Give an initial vector $x^{0}$, and parameters $\omega$ and $\hat{\omega}$.}
\STATE{Set $\bar{b}_1=A_1^Tb_1 $, $R=A_{2}^TA_{2} $, $P=\left(A_{1}^TA_{1}\right)^{-1} $ and $\tau=\omega+\hat{\omega}-\omega\hat{\omega}$.}
\STATE{Compute $\bar{\delta}_{1}^{0}=A_{1}^{\top}\left(b_{1}-A_{1} x^{0}\right), \delta_{2}^{0}=b_{2}-A_{2} x^{0}$.}
\FOR{$k=1,2, \ldots$ until convergence,}
\STATE{$\bar{\delta}_{1}^{k+1}  =\tau A_{2}^{\top}\left[(1-\omega) \delta_{2}^{k}+\omega b_{2}\right]+\omega \tau R P\left(\bar{\delta}_{1}^{k}-\bar{b}_{1}\right)+(1-\tau) \bar{\delta}_{1}^{k}$.}
\STATE{$x^{k+1}  =(1-\tau) x^{k}+P\left[\tau \bar{b}_{1}-\omega(1-\hat{\omega}) \bar{\delta}_{1}^{k}-\hat{\omega }\bar{\delta}_{1}^{k+1}\right]$. }
\STATE{$\delta_{2}^{k+1} =(1-\tau)\left(A_{2} x^{k}+\delta_{2}^{k}\right)-A_{2} x^{k+1}+\tau b_{2}$. }
\ENDFOR
\end{algorithmic}
\end{algorithm}

In \Cref{alg:USSOR}, the steps 2 and 3 need operation counts of about $mn^2+2n^3+2pn+3-n$ and $2mn+2pn-n$, respectively, and hence give the total counts of about $mn^2+2n^3+2mn+4pn+3-2n$. Determining $\bar{\delta}_{1}^{k+1}$, $x^{k+1}$ and $\delta_{2}^{k+1}$ needs to compute step 5 to step 7, which requires operation counts of about $2qn+4n^2+3q+3n+3$, $2n^2+6n+3$, and $4qn+3q+1$, respectively, and hence gives the total counts of about $6qn+6n^2+6q+9n+7$. Then the total operation counts of the USSOR method are about $$mn^2+2n^3+2mn+4pn+3-2n+\left( 6qn+6n^2+6q+9n+7\right)\cdot T_{\text{USSOR}},$$ where $T_{\text{USSOR}}$ is the number of iterations of the USSOR method.

For the SP and SP-RK-RGS methods,  from \Cref{remark-cost: SP,remark-cost: SP-RK-RGS}, we know that they require operation counts of about $mn^2+4n^3+2mn+2n^2-2n+2n^2 \cdot T_{\text{SP}}$ and $qn^2+2mn-n+\left(2n^2+\left(2pn+6p+2\right)\cdot T_{\text{RK-RGS}} \right)\cdot T_{\text{SP-RK-RGS}},$ respectively. The differences among them and the cost of the USSOR method are also introduced in \Cref{remark-cost: SP,remark-cost: SP-RK-RGS}, respectively.

For the  SP-SCD method, since the sampling probability (\ref{Sec4.111}) 
is computationally prohibitive, we rewrite \Cref{alg:SP-SCD} as \Cref{alg:SP-SCD-version2} and apply it to the specific experiments. For simplicity for analyzing the computational complexity, we assume that $\alpha_t=\alpha_{t+1}=\alpha$ for $t=0, 1, 2, \ldots $ and the inner SCD update has the same iteration numbers $T_{\text{SCD}}$ for $k=0, 1, 2, \ldots. $ In this case, the total cost of \Cref{alg:SP-SCD-version2} is about
$$ mn^2+2mn-n+ \left(2n^2+ \left(2n+2\alpha+4 \right)\cdot T_{\text{SCD}}\right)\cdot T_{\text{SP-SCD}},$$
where $T_{\text{SP-SCD}}$ is the outer iteration numbers of the SP-SCD method. This cost is almost the same as the ones of the USSOR and SP methods. However, the SP-SCD method performs best in numerical experiments. This is because the total operation counts of various methods given above are only approximate and may be far from the accurate ones. One of the contributing factors is 
the actual iteration numbers. This implies that the above distinguishing on complexities of the four methods is quite wild and hence may only provide limited suggestions for practical applications.

\begin{algorithm}
\caption{SP-SCD method for ILS problem (\ref{Sec11}).}
\label{alg:SP-SCD-version2}
\begin{algorithmic}[1]
\STATE{Input: $A$, $J$, $b$, and initial estimate $x_0$. }
\STATE{Set $ \bar{A_1}=A_1^TA_1$, $ \bar{A_2}=A_2^TA_2$, and $\bar{b}=A^TJb$.}
\FOR{$k=0, 1, 2, \ldots $ until convergence,}
\STATE{Compute $\hat{b}=\bar{A_2}x_{k}+\bar{b}$.}
\STATE{Set $\beta_0=0$ and $r_0=\hat{b}-\bar{A_1}\beta_0$.}
\FOR{$t=0, 1, 2, \ldots $ until convergence,}
\STATE{Generate a positive integer $\alpha_t\in[n]$ at random.}
\STATE{Choose an index subset of size $\alpha_t$, $\tau_t$, uniformly at random from among $ [n]$.  }
\STATE{Set $j_{t}=\text{arg}\max\limits_{s\in\tau_t}\left|r^{(s)} \right|^2$.}
\STATE{Update $r_{t+1}=r_{t}-\frac{ r_{k}^{\left(j_{t}\right)}}{ \bar{A_1}_{\left(j_{t}, j_{t}\right)}}  \bar{A_1}_{(j_{t})}  $.}
\STATE{Update $\beta_{t+1}=\beta_{t}+\frac{ r_{k}^{\left(j_{t}\right)}}{ \bar{A_1}_{\left(j_{t}, j_{t}\right)}}  e_{(j_{t})}  $.}
\ENDFOR
\STATE{Set $x_{k+1}=\beta_{t+1}.$ }
\ENDFOR
\end{algorithmic}
\end{algorithm}

\subsection{Examples from \citet{song2020ussor}} 
\label{subsec: random data matrices}
Specifically, we set
$A_1=\texttt{rand(p,n)}$, $A_2=7*\texttt{eye(q,n)}$, 
$b_1=\texttt{rand(p,1)}$, and $b_2=\texttt{rand(q,1)}$. 
For the optimal parameters of the USSOR method, we obtain them 
according to Theorem 3.1 in \cite{song2020ussor}. Numerical results on different $p,q$ and $n$ are reported in Tables \ref{table-p-30000-q-n} and \ref{table-p-40000-q-n}.

\begin{table}[t!]
\tblcaption{ Numerical results of the methods on $p=30000$ and $q=n$.}
{%
\begin{tabular}{@{}ccccc@{}}
\tblhead{&  $m\times n$   &  $43000\times13000$   & $44000\times14000$  &$45000\times15000$  }
\multirow{6}{*}{USSOR} & $\tau $        &  1.0458      & 1.0544     &1.0645      \cr
                       & $\omega$       &  0.5000      & 0.5000     &0.5000        \cr
                       & $\hat{\omega}$ &  1.0917      & 1.1087     &1.1291      \cr
                       & IT             &  3           & 4          &4            \cr
                       & CPU            &  1974.0      & 2504.4     &2938.5      \cr

\hline
\multirow{3}{*}{SP}    & IT             &  1          & 1          &1               \cr
                       & CPU            &  559.3      &619.1       &830.2         \cr
                       & \texttt{speed-up-1}     &  3.5294     &  4.0452   & 3.5395   \cr
\hline
\multirow{5}{*}{SP-RK-RGS}&  IT-inner      &$1.5339\times 10^{5}$    &$1.6681\times 10^{5}$   &$ 2.3281\times 10^{5}$    \cr
& IT             &  1.3000       &1.3000    &1.7000      \cr

    & CPU-inner    & 283.8500  &316.6344  &452.3172   \cr
                       & CPU            & 514.6313 & 592.2625  &803.0406  \cr

                       & \texttt{speed-up-2}           &  3.8358  &  4.2285   & 3.6592   \cr
\hline
\multirow{5}{*}{SP-SCD}
&  IT-inner      &$1.1343\times 10^{4}$  &$1.3420\times 10^{4}$  &$1.3618\times 10^{4}$   \cr
& IT             &  1.1000    &1.3000   & 1.1000   \cr
& CPU-inner    &17.7969  & 22.0516 &  23.7906   \cr
& CPU            & 227.7141  &286.9359  &324.5609   \cr

 &\texttt{speed-up-3}          &  8.6688  &  8.7281   & 9.0538
\lastline
\end{tabular}
}
\label{table-p-30000-q-n}
\end{table}

\begin{table}[t!]
\tblcaption{ Numerical results of the methods on $p=40000$ and $q=n$.}
{%
\begin{tabular}{@{}ccccc@{}}
\tblhead{ &  $m\times n$   &  $53000\times13000$   & $54000\times14000$  &$55000\times15000$    }
\multirow{6}{*}{USSOR} & $\tau $       &     1.0206 &  1.0230 &    1.0258  \cr
                       & $\omega$       &  0.5000      & 0.5000     &0.5000       \cr
                       & $\hat{\omega}$ &    1.0412 &  1.0460    &  1.0516      \cr
                       & IT            &   3     &   3    &  3     \cr
                       & CPU         &    2581.8  &  2948.1 &   3392.9      \cr

\hline
\multirow{3}{*}{SP}    & IT       &1    & 1  &   1    \cr
                       & CPU      &599.1    &762.0   & 927.1    \cr
                       &\texttt{speed-up-1}     &  4.3095 &   3.8689  &  3.6597      \cr
\hline
\multirow{5}{*}{SP-RK-RGS}   &  IT-inner     &  $1.2676\times 10^{5}$    & $1.3276\times 10^{5}$    & $1.4641\times 10^{5}$    \cr
& IT        & 1.1000   & 1.1000   & 1.2000     \cr
   & CPU-inner    &      254.0328 & 271.6281  &309.8109 \cr
                       & CPU     &  546.1656  &591.3906 & 708.3609 \cr

                       & \texttt{speed-up-2}          & 4.7271   & 4.9850 &   4.7898    \cr
\hline
\multirow{5}{*}{SP-SCD}&  IT-inner      &   10120  &  9208   & 9165   \cr
& IT          &   1.2000   & 1.1000  &  1.0000  \cr
& CPU-inner   &      15.7312  & 15.1938 &  15.9953  \cr
& CPU           &     272.8125 & 310.7219 & 359.2063  \cr

 &\texttt{speed-up-3}         &   9.4636   & 9.4879   & 9.4455
\lastline
\end{tabular}
}
\label{table-p-40000-q-n}
\end{table}

From these two tables, we can find that our proposed three methods outperform the USSOR method in terms of the iteration numbers and computing time, 
and the computing time speed-up is at least 3.5294 (see \texttt{speed-up-1} in Table \ref{table-p-30000-q-n} for the $43000\times 13000$ matrix). 
Meanwhile, the SP-RK-RGS and SP-SCD methods are more efficient than the SP method in computing time, and the efficiency of the SP-SCD method is the most remarkable. 
This is probably mainly because the inner iteration of the SP-SCD method needs fewer iteration numbers and less running time compared with the one of the SP-RK-RGS method.

\subsection{Examples from 
Minkowski spaces}
\label{subsec:  Minkowski spaces}

In this case, $p=m-1$ and $q=1$. We consider the same setting as in 
\Cref{subsec: random data matrices}. That is, 
$A_1=\texttt{rand(p,n)}$, $A_2=7*\texttt{eye(1,n)}$, 
$b_1=\texttt{rand(p,1)}$, and $b_2=\texttt{rand(1,1)}$. The optimal parameters of the USSOR method are also computed  
according to Theorem 3.1 in \cite{song2020ussor}. We report the numerical results on different $p$ and $n$ in Tables \ref{table-p-50000-q-1} and \ref{table-p-60000-q-1}, which show the similar results obtained in \Cref{subsec: random data matrices}. That is, the SP, SP-RK-RGS and SP-SCD methods outperform the USSOR method in both iteration numbers and CPU time, and the SP-RK-RGS and SP-SCD methods have better performance in computing time.

\begin{table}[t!]
\tblcaption{ Numerical results of the methods in Minkowski spaces with $p=50000$ and $q=1$.}
{%
\begin{tabular}{@{}ccccc@{}}
\tblhead{  &  $m\times n$   &  $50001\times13000$   & $50001\times14000$  &$50001\times15000$     }
\multirow{6}{*}{USSOR} & $\tau $        & 1.0040    & 1.0041   & 1.0043   \cr
                       & $\omega$       &  0.5000      & 0.5000     &0.5000        \cr
                       & $\hat{\omega}$ &     1.0080   & 1.0083   & 1.0085    \cr
                       & IT            &      2   &  2  &   2  \cr
                       & CPU           &     680.7    & 864.4  &  1048.8    \cr
\hline
\multirow{3}{*}{SP}    & IT         &     1   &  1  &   1  \cr
                       & CPU       &  566.9453   &731.0016   & 881.0750  \cr
                       & \texttt{speed-up-1}          &     1.2006 &   1.1825 &   1.1904   \cr
\hline
\multirow{5}{*}{SP-RK-RGS} &  IT-inner    & $ 1.0920\times 10^{5}$    & $1.1956\times 10^{5}$   & $ 1.1491\times 10^{5}$        \cr
& IT          &   1     &   1    &   1     \cr
  & CPU-inner   &        237.4219 & 265.5109 & 264.4844      \cr
                       & CPU         &    466.9094&  530.9078&  560.6984      \cr

                       &  \texttt{speed-up-2}          &    1.4579 &   1.6282 &   1.8705    \cr
\hline
\multirow{5}{*}{SP-SCD}  &  IT-inner    &     6749.2   & 7315.2  &  7398.8        \cr
& IT         &    1    &  1     &   1    \cr
 & CPU-inner   &   11.7109  & 15.6125  & 16.3172    \cr
& CPU          &   198.2734 & 231.1078 & 265.1656     \cr

 &\texttt{speed-up-3}         &     3.4331  &  3.7402 &   3.9553
\lastline
\end{tabular}
}
 \label{table-p-50000-q-1}
\end{table}

\begin{table}[t!]
\tblcaption{ Numerical results of the methods in Minkowski spaces with $p=60000$ and $q=1$.}
{%
\begin{tabular}{@{}ccccc@{}}
\tblhead{  &  $m\times n$   &  $60001\times13000$   & $60001\times14000$  &$60001\times15000$      }
\multirow{6}{*}{USSOR} & $\tau $        &     1.0031  &  1.0032 &   1.0033    \cr
                       & $\omega$       &  0.5000      & 0.5000     &0.5000      \cr
                       & $\hat{\omega}$ &     1.0063 &   1.0064  &  1.0066 \cr
                       & IT            &   2    & 2    & 2      \cr
                       & CPU           &   739.6   & 901.1  &  1152.7    \cr

\hline
\multirow{3}{*}{SP}    & IT         &     1   &   1    &     1   \cr
                       & CPU      &      633.1  &  768.9  &  1019.6       \cr
                       &\texttt{speed-up-1}         &     1.1681  &  1.1719   & 1.1305     \cr
\hline
\multirow{5}{*}{SP-RK-RGS} &  IT-inner   & $1.0131\times 10^{5}$    & $ 1.0289\times 10^{5}$    & $ 1.1229\times 10^{5}$   \cr
& IT        &    1   &  1   &  1     \cr
   & CPU-inner   &     236.9672 & 245.5719 & 279.6359    \cr
                       & CPU         &   511.2844  &567.4719&  687.4875   \cr

                       & \texttt{speed-up-2}         &    1.4465   & 1.5879 &   1.6766     \cr
\hline
\multirow{5}{*}{SP-SCD} &  IT-inner    &   5584.7  &  6046.7   & 6026.5     \cr
& IT       &    1   &  1    & 1\cr
  & CPU-inner   &      13.1734  & 14.0109   &10.6609 \cr
& CPU          &       238.0406 &  273.4203  &310.2516     \cr

 &\texttt{speed-up-3}        &     3.1068   & 3.2956  &  3.7153
\lastline
\end{tabular}
}
\label{table-p-60000-q-1}
\end{table}

\section{Concluding remarks}
\label{sec:conclusions}

In this paper, we propose three `unconditionally' convergent iterative methods, i.e., the SP, SP-RK-RGS, and SP-SCD methods, to solve the ILS problem (\ref{Sec11}).
Numerical results show that they all have quite decent performance, and the two randomized methods are particularly efficient in computing time. 
A future work is to consider the splitting-based randomized iterative methods for the large-scale ILS problem with equality constraints \citep[see, e.g.,][]{Bojanczyk2003, Liu2010,Mastronardi2014,mastronardi2015structurally}.

%
%

\section*{Funding}
The National Natural Science Foundation of China (No. 11671060); The Natural Science Foundation of Chongqing, China (No. cstc2019jcyj-msxmX0267).

\bibliographystyle{IMANUM-BIB}
\bibliography{IMANUM-refs}

\begin{thebibliography}{}

\bibitem[Bai \& Wu(2018)Bai \& Wu]{bai2018greedy}
{\sc Bai, Z.~Z. \& Wu, W.~T.} (2018)
\newblock {On greedy randomized Kaczmarz method for solving large sparse linear
  systems}.
\newblock {\em SIAM J. Sci. Comput.}, {\bf 40}, A592--A606.

\bibitem[Bojanczyk {\em et~al.}(2003a)Bojanczyk, Higham, \&
  Patel]{bojanczyk2003solving}
{\sc Bojanczyk, A.~W., Higham, N.~J. \& Patel, H.} (2003a)
\newblock {Solving the indefinite least squares problem by hyperbolic QR
  factorization}.
\newblock {\em SIAM J. Matrix Anal. Appl.}, {\bf 24}, 914--931.

\bibitem[Bojanczyk {\em et~al.}(2003b)Bojanczyk, Higham, \&
  Patel]{Bojanczyk2003}
{\sc Bojanczyk, A.~W., Higham, N.~J. \& Patel, H.} (2003b)
\newblock {The equality constrained indefinite least squares problem: theory
  and algorithms}.
\newblock {\em BIT Numer. Math.}, {\bf 43}, 505--517.

\bibitem[Bojanczyk(2021)Bojanczyk]{bojanczyk2021algorithms}
{\sc Bojanczyk, A.~W.} (2021)
\newblock Algorithms for indefinite linear least squares problems.
\newblock {\em Linear Algebra Appl.}, {\bf 623}, 104--127.

\bibitem[Chandrasekaran {\em et~al.}(1998)Chandrasekaran, Gu, \&
  Sayed]{chandrasekaran1998stable}
{\sc Chandrasekaran, S., Gu, M. \& Sayed, A.~H.} (1998)
\newblock A stable and efficient algorithm for the indefinite linear
  least-squares problem.
\newblock {\em SIAM J. Matrix Anal. Appl.}, {\bf 20}, 354--362.

\bibitem[De~Loera {\em et~al.}(2017)De~Loera, Haddock, \&
  Needell]{de2017sampling}
{\sc De~Loera, J.~A., Haddock, J. \& Needell, D.} (2017)
\newblock {A sampling Kaczmarz-Motzkin algorithm for linear feasibility}.
\newblock {\em SIAM J. Sci. Comput.}, {\bf 39}, S66--S87.

\bibitem[Diao \& Zhou(2019)Diao \& Zhou]{diao2019backward}
{\sc Diao, H.~A. \& Zhou, T.~Y.} (2019)
\newblock Backward error and condition number analysis for the indefinite
  linear least squares problem.
\newblock {\em Int. J. Comput. Math.}, {\bf 96}, 1603--1622.

\bibitem[Du {\em et~al.}(2020)Du, Si, \& Sun]{du2020randomized}
{\sc Du, K., Si, W.~T. \& Sun, X.~H.} (2020)
\newblock {Randomized extended average block Kaczmarz for solving least
  squares}.
\newblock {\em SIAM J. Sci. Comput.}, {\bf 42}, A3541--A3559.

\bibitem[Golub \& Van~Loan(1980)Golub \& Van~Loan]{golub1980analysis}
{\sc Golub, G.~H. \& Van~Loan, C.~F.} (1980)
\newblock An analysis of the total least squares problem.
\newblock {\em SIAM J. Numer. Anal.}, {\bf 17}, 883--893.

\bibitem[Gower {\em et~al.}(2021)Gower, Molitor, Moorman, \&
  Needell]{gower2021adaptive}
{\sc Gower, R.~M., Molitor, D., Moorman, J. \& Needell, D.} (2021)
\newblock On adaptive sketch-and-project for solving linear systems.
\newblock {\em SIAM J. Matrix Anal. Appl.}, {\bf 42}, 954--989.

\bibitem[Haddock \& Ma(2021)Haddock \& Ma]{haddock2021greed}
{\sc Haddock, J. \& Ma, A.} (2021)
\newblock {Greed works: an improved analysis of sampling Kaczmarz-Motzkin}.
\newblock {\em SIAM J. Math. Data Sci.}, {\bf 3}, 342--368.

\bibitem[Hassibi {\em et~al.}(1993)Hassibi, Sayed, \&
  Kailath]{hassibi1993recursive}
{\sc Hassibi, B., Sayed, A.~H. \& Kailath, T.} (1993)
\newblock {Recursive linear estimation in Krein spaces. I. Theory}.
\newblock {\em Proceedings of 32nd IEEE Conference on Decision and Control\/}.
\newblock IEEE, IEEE, pp. 3489--3494.

\bibitem[Hefny {\em et~al.}(2017)Hefny, Needell, \& Ramdas]{hefny2017rows}
{\sc Hefny, A., Needell, D. \& Ramdas, A.} (2017)
\newblock {Rows versus columns: randomized Kaczmarz or Gauss-Seidel for ridge
  regression}.
\newblock {\em SIAM J. Sci. Comput.}, {\bf 39}, S528--S542.

\bibitem[Jiao {\em et~al.}(2017)Jiao, Jin, \& Lu]{jiao2017preasymptotic}
{\sc Jiao, Y.~L., Jin, B.~T. \& Lu, X.~L.} (2017)
\newblock {Preasymptotic convergence of randomized Kaczmarz method}.
\newblock {\em Inverse Problems\/}, {\bf 33}, 125012.

\bibitem[Leventhal \& Lewis(2010)Leventhal \& Lewis]{leventhal2010randomized}
{\sc Leventhal, D. \& Lewis, A.~S.} (2010)
\newblock Randomized methods for linear constraints: convergence rates and
  conditioning.
\newblock {\em Math. Oper. Res.}, {\bf 35}, 641--654.

\bibitem[Li {\em et~al.}(2014)Li, Wang, \& Yang]{li2014mixed}
{\sc Li, H.~Y., Wang, S.~X. \& Yang, H.} (2014)
\newblock On mixed and componentwise condition numbers for indefinite least
  squares problem.
\newblock {\em Linear Algebra Appl.}, {\bf 448}, 104--129.

\bibitem[Li \& Wang(2018)Li \& Wang]{li2018partial}
{\sc Li, H.~Y. \& Wang, S.~X.} (2018)
\newblock On the partial condition numbers for the indefinite least squares
  problem.
\newblock {\em Appl. Numer. Math.}, {\bf 123}, 200--220.

\bibitem[Lin {\em et~al.}(2015)Lin, Zang, \& Qing]{lin2015extended}
{\sc Lin, C., Zang, J.~F. \& Qing, A.~Y.} (2015)
\newblock {Extended Kaczmarz algorithm with projection adjustment}.
\newblock {\em 2015 IEEE MTT-S International Conference on Numerical
  Electromagnetic and Multiphysics Modeling and Optimization\/}.
\newblock IEEE, IEEE, pp. 1--3.

\bibitem[Liu \& Wright(2016)Liu \& Wright]{liu2016accelerated}
{\sc Liu, J. \& Wright, S.~J.} (2016)
\newblock {An accelerated randomized Kaczmarz algorithm}.
\newblock {\em Math Comp.}, {\bf 85}, 153--178.

\bibitem[Liu \& Li(2011)Liu \& Li]{liu2011preconditioned}
{\sc Liu, Q.~H. \& Li, X.~J.} (2011)
\newblock Preconditioned conjugate gradient methods for the solution of
  indefinite least squares problems.
\newblock {\em Calcolo\/}, {\bf 48}, 261--271.

\bibitem[Liu \& Liu(2014)Liu \& Liu]{liu2014block}
{\sc Liu, Q.~H. \& Liu, A.~J.} (2014)
\newblock {Block SOR methods for the solution of indefinite least squares
  problems}.
\newblock {\em Calcolo\/}, {\bf 51}, 367--379.

\bibitem[Liu \& Wang(2010)Liu \& Wang]{Liu2010}
{\sc Liu, Q.~H. \& Wang, M.~H.} (2010)
\newblock {Algebraic properties and perturbation results for the indefinite
  least squares problem with equality constraints}.
\newblock {\em Int. J. Comput. Math\/}, {\bf 87}, 425--434.

\bibitem[Liu \& Zhang(2013)Liu \& Zhang]{liu2013incomplete}
{\sc Liu, Q.~H. \& Zhang, F.~D.} (2013)
\newblock {Incomplete hyperbolic Gram-Schmidt-based preconditioners for the
  solution of large indefinite least squares problems}.
\newblock {\em J. Comput. Appl. Math.}, {\bf 250}, 210--216.

\bibitem[Liu \& Gu(2019)Liu \& Gu]{liu2019variant}
{\sc Liu, Y. \& Gu, C.~Q.} (2019)
\newblock {Variant of greedy randomized Kaczmarz for ridge regression}.
\newblock {\em Appl. Numer. Math.}, {\bf 143}, 223--246.

\bibitem[Ma {\em et~al.}(2015)Ma, Needell, \& Ramdas]{ma2015convergence}
{\sc Ma, A., Needell, D. \& Ramdas, A.} (2015)
\newblock {Convergence properties of the randomized extended Gauss-Seidel and
  Kaczmarz methods}.
\newblock {\em SIAM J. Matrix Anal. Appl.}, {\bf 36}, 1590--1604.

\bibitem[Ma {\em et~al.}(2018)Ma, Needell, \& Ramdas]{ma2018iterative}
{\sc Ma, A., Needell, D. \& Ramdas, A.} (2018)
\newblock {Iterative methods for solving factorized linear systems}.
\newblock {\em SIAM J. Matrix Anal. Appl.}, {\bf 39}, 104--122.

\bibitem[Mastronardi \& {Van Dooren}(2014)Mastronardi \& {Van
  Dooren}]{Mastronardi2014}
{\sc Mastronardi, N. \& {Van Dooren}, P.} (2014)
\newblock {An algorithm for solving the indefinite least squares problem with
  equality constraints}.
\newblock {\em BIT Numer. Math.}, {\bf 54}, 201--218.

\bibitem[Mastronardi \& Van~Dooren(2015)Mastronardi \&
  Van~Dooren]{mastronardi2015structurally}
{\sc Mastronardi, N. \& Van~Dooren, P.} (2015)
\newblock A structurally backward stable algorithm for solving the indefinite
  least squares problem with equality constraints.
\newblock {\em IMA J. Numer. Anal.}, {\bf 35}, 107--132.

\bibitem[Morshed {\em et~al.}(2020)Morshed, Islam, \&
  Noor-E-Alam]{morshed2020accelerated}
{\sc Morshed, M.~S., Islam, M.~S. \& Noor-E-Alam, M.} (2020)
\newblock {Accelerated sampling Kaczmarz Motzkin algorithm for the linear
  feasibility problem}.
\newblock {\em J. Global Optim.}, {\bf 77}, 361--382.

\bibitem[Morshed {\em et~al.}(2021)Morshed, Islam, \&
  Noor-E-Alam]{morshed2021sampling}
{\sc Morshed, M.~S., Islam, M.~S. \& Noor-E-Alam, M.} (2021)
\newblock {Sampling Kaczmarz-Motzkin method for linear feasibility problems:
  generalization and acceleration}.
\newblock {\em Math. Program.}, 1--61.

\bibitem[Motzkin \& Schoenberg(1954)Motzkin \&
  Schoenberg]{motzkin_schoenberg_1954}
{\sc Motzkin, T.~S. \& Schoenberg, I.~J.} (1954)
\newblock The relaxation method for linear inequalities.
\newblock {\em Canad. J. Math.}, {\bf 6}, 393--404.

\bibitem[Necoara(2019)Necoara]{necoara2019faster}
{\sc Necoara, I.} (2019)
\newblock {Faster randomized block Kaczmarz algorithms}.
\newblock {\em SIAM J. Matrix Anal. Appl.}, {\bf 40}, 1425--1452.

\bibitem[Needell \& Tropp(2014)Needell \& Tropp]{needell2014paved}
{\sc Needell, D. \& Tropp, J.~A.} (2014)
\newblock {Paved with good intentions: analysis of a randomized block Kaczmarz
  method}.
\newblock {\em Linear Algebra Appl.}, {\bf 441}, 199--221.

\bibitem[Niu \& Zheng(2020)Niu \& Zheng]{niu2020greedy}
{\sc Niu, Y.~Q. \& Zheng, B.} (2020)
\newblock {A greedy block Kaczmarz algorithm for solving large-scale linear
  systems}.
\newblock {\em Appl. Math. Lett.}, {\bf 104}, 106294.

\bibitem[Nutini {\em et~al.}(2016)Nutini, Sepehry, Virani, Laradji, Schmidt, \&
  Koepke]{nutini2016convergence}
{\sc Nutini, J., Sepehry, B., Virani, A., Laradji, I., Schmidt, M. \& Koepke,
  H.} (2016)
\newblock {Convergence rates for greedy Kaczmarz algorithms}.
\newblock {\em 32nd Conference on Uncertainty in Artificial Intelligence\/}.
\newblock AUAI Press, AUAI Press.

\bibitem[{\v{S}}ego(2009){\v{S}}ego]{vsego2009two}
{\sc {\v{S}}ego, V.} (2009)
\newblock {Two-Sided Hyperbolic Singular Value Decomposition}.
\newblock {\em Ph.D. thesis}, Department of Mathematics, University of Zagreb.

\bibitem[Song(2020)Song]{song2020ussor}
{\sc Song, J.} (2020)
\newblock {USSOR method for solving the indefinite least squares problem}.
\newblock {\em Int. J. Comput. Math.}, {\bf 97}, 1781--1791.

\bibitem[Strohmer \& Vershynin(2009)Strohmer \&
  Vershynin]{strohmer2009randomized}
{\sc Strohmer, T. \& Vershynin, R.} (2009)
\newblock {A randomized Kaczmarz algorithm with exponential convergence}.
\newblock {\em J. Fourier Anal. Appl.}, {\bf 15}, 262--278.

\bibitem[Van~Huffel \& Vandewalle(1991)Van~Huffel \& Vandewalle]{van1991total}
{\sc Van~Huffel, S. \& Vandewalle, J.} (1991)
\newblock {\em {The Total Least Squares Problem: Computational Aspects and
  Analysis}\/}.
\newblock Philadelphia: SIAM.

\bibitem[Wang {\em et~al.}(2015)Wang, Agaskar, \& Lu]{wang2015randomized}
{\sc Wang, C., Agaskar, A. \& Lu, Y.~M.} (2015)
\newblock {Randomized Kaczmarz algorithm for inconsistent linear systems: an
  exact MSE analysis}.
\newblock {\em 2015 International Conference on Sampling Theory and
  Applications\/}.
\newblock IEEE, IEEE, pp. 498--502.

\bibitem[Xu(2004)Xu]{xu2004backward}
{\sc Xu, H.~G.} (2004)
\newblock {A backward stable hyperbolic QR factorization method for solving
  indefinite least squares problem}.
\newblock {\em J. Shanghai Univ.}, {\bf 8}, 391--396.

\bibitem[Zhang \& Li(2021)Zhang \& Li]{zhang2021block}
{\sc Zhang, Y.~J. \& Li, H.~Y.} (2021)
\newblock {Block sampling Kaczmarz-Motzkin methods for consistent linear
  systems}.
\newblock {\em Calcolo\/}, {\bf 58}, 1--20.

\bibitem[Zhang \& Li(2022)Zhang \& Li]{Zhang2022MK}
{\sc Zhang, Y.~J. \& Li, H.~Y.} (2022)
\newblock {Greedy Motzkin-Kaczmarz methods for solving linear systems}.
\newblock {\em Numer. Linear Algebra Appl.}, {\bf 29}, e2429.

\bibitem[Zouzias \& Freris(2013)Zouzias \& Freris]{zouzias2013randomized}
{\sc Zouzias, A. \& Freris, N.~M.} (2013)
\newblock {Randomized extended Kaczmarz for solving least squares}.
\newblock {\em SIAM J. Matrix Anal. Appl.}, {\bf 34}, 773--793.

\end{thebibliography}

\end{document}